\documentclass[11pt]{amsart}
\RequirePackage[l2tabu, orthodox]{nag}
\usepackage[T1]{fontenc}
\usepackage[utf8]{inputenc}
\usepackage{lmodern}
\usepackage[english]{babel}
\usepackage{tikz-cd} 
\usepackage{amsmath}
\usepackage{amsfonts}
\usepackage{amssymb}
\usepackage{amsrefs,amsthm}
\usepackage{enumerate}
\usepackage{xcolor}
\usepackage[normalem]{ulem}

\usepackage[margin=2.5cm]{geometry}
\newtheorem{theorem}{Theorem}[section]
\newtheorem{lemma}[theorem]{Lemma}
\newtheorem{claim}{Claim}
\newtheorem*{claim*}{Claim}
\newtheorem{theoremA}{Theorem}

\newtheorem{corA}[theoremA]{Corollary}
\newtheorem*{theorem*}{Theorem}

\newtheorem{proposition}[theorem]{Proposition}
\newtheorem{question}[theorem]{Question}

\theoremstyle{definition}
\newtheorem{definition}[theorem]{Definition}

\theoremstyle{remark}

\newtheorem{remark}[theorem]{Remark}

\numberwithin{equation}{section}
\numberwithin{figure}{section}

\def\NN{{\mathbb N}}
\def\QQ{{\mathbb Q}}
\def\RR{{\mathbb R}}
\newcommand{\T}{\mathbb{S}^1}
\def\ZZ{{\mathbb Z}}

\def\idid{{\mathrm{id}}}
\def\aff{{\mathrm{Aff}}}
\def\homeo{{\mathrm{Homeo}}}

\def\psl{{\mathrm{PSL}}}
\def\SO{{\mathrm{SO}}}
\def\fix{{\mathrm{Fix}}}
\def\stab{{\mathrm{Stab}}}
\def\rot{{\mathrm{rot}}}
\def\supp{{\mathrm{supp}}}
\renewcommand{\setminus}{\smallsetminus}

\def\kovacevic{{Kova\v{c}evi\'{c}}}

\usepackage{indentfirst}

\usepackage[colorlinks=true,linkcolor=blue,citecolor=magenta]{hyperref}

\title{Non-locally discrete actions on the circle with at most $N$ fixed points}

\author{Christian Bonatti \and Jo\~ao Carnevale \and Michele Triestino}
\date{\today}

\begin{document}

\begin{abstract}
	A subgroup of $\mathrm{Homeo}_+(\mathbb{S}^1)$ is M\"obius-like if every element is conjugate to an element of $\mathrm{PSL}(2,\mathbb{R})$. In general, a M\"obius-like subgroup of $\mathrm{Homeo}_+(\mathbb{S}^1)$ is not necessarily (semi-)conjugate to a subgroup of $\mathrm{PSL}(2,\mathbb{R})$, as discovered by N. Kova\v{c}evi\'{c} [Trans.\ Amer.\ Math.\ Soc.\ \textbf{351} (1999), 4823--4835].
	Here we determine simple dynamical criteria for the existence of such a (semi-)conjugacy.
	We show that M\"obius-like subgroups of $\mathrm{Homeo}_+(\mathbb{S}^1)$ which are elementary (namely, preserving a Borel probability measure), are semi-conjugate to subgroups of $\mathrm{PSL}(2,\mathbb{R})$.
	On the other hand, we provide an example of elementary subgroup of $\mathrm{Diff}^\infty_+(\mathbb{S}^1)$ satisfying that every non-trivial element fixes at most 2 points, which is not isomorphic to any subgroup of $\mathrm{PSL}(2,\mathbb{R})$. Finally, we show that non-elementary, non-locally discrete subgroups acting with at most $N$ fixed points are conjugate to a dense subgroup of some finite central extension of $\mathrm{PSL}(2,\mathbb{R})$.
	
		\smallskip
	
\end{abstract}

\maketitle

\section{Introduction}

	Let $X$ be a topological space, and $G\le \homeo(X)$ a subgroup of homeomorphisms of $X$. Given $N\in\NN$, we say that $G$ \emph{has at most $N$ fixed points} if every non-trivial element of $G$ has at most $N$ fixed points.
When $N=0$, this means that the action of $G$ on $X$ is free.
A natural problem is to characterize actions with at most $N$ fixed points: 
\begin{question}\label{q.general}
	Given $N\in\NN$, which subgroups of $\homeo(X)$ have at most $N$ fixed points?
\end{question}

Here we will address this problem for the case when $X=\T$ is the circle (the case  $N=2$ and $X=\RR$ has been discussed by the second author in \cite{carnevale2022groups}, extending classical work of H\"older and Solodov). To simplify the discussion, we will rather consider subgroups of $\homeo_+(\T)$, the group of order-preserving homeomorphisms of the circle, which is of index 2 in $\homeo(\T)$. The classical group $\psl(2,\mathbb{R})$ of Möbius transformations acts on the circle with at most $2$ fixed points. This group may be seen in two ways:
\begin{itemize}
	\item it is the group of projective transformations of the projective line $\mathbb{R}\mathrm P^1$;
	\item it is also the group of isometries of the hyperbolic disc $\mathbb{D}$, which acts on its circle at infinity. 
\end{itemize}
For any integer $k\ge 1$, the $k$-fold central extension
\[1\to \ZZ_k\to \psl^{(k)}(2,\RR)\to \psl(2,\RR)\] 
also acts on the circle (which is homeomorphic to its $k$-fold covering), and it gives a natural example of group acting with at most $2k$ fixed points.
Somehow we want to understand to what extent groups with at most $2k$ fixed points are comparable to subgroups of $\psl^{(k)}(2,\mathbb{R})$. The case $k=1$ is already interesting on its own. For this, we say that a group $G$ is \emph{M\"obius-like} if every element is individually conjugate to an element in $\psl(2,\RR)$. In practice, this means that any element in $G$ is either conjugate to a rotation (\emph{elliptic}), or it admits exactly one fixed point (\emph{parabolic}), or it admits exactly two fixed points, one attracting and one repelling (\emph{hyperbolic}). On the other hand, a basic example of subgroup of $\homeo(\T)$ with at most 2 fixed points but not M\"obius-like, is the cyclic subgroup generated by a homeomorphism with exactly two fixed points, but which is not (topologically) hyperbolic: in this case both fixed points are (topologically) \emph{parabolic}, namely repelling on one side and attracting on the other.

Several works have exhibited conditions under which M\"obius-like groups are conjugate to subgroups of $\psl(2,\RR)$. Notably, a major result in this context has been obtained by Tukia \cite{Tukia} (in the torsion-free case), and then independently extended by Gabai \cite{Gabai} and Casson and Jungreis \cite{Casson-Jungreis}, after which subgroups of $\psl(2,\RR)$ are dynamical determined through the notion of convergence groups.
Let us also mention the more recent work of Baik \cite{Baik}, describing a characterization of subgroups of $\psl(2,\RR)$ in terms of invariant laminations on the circle.
In \cite{Ko1}, Kova\v{c}evi\'{c} proved that a M\"obius-like subgroup of $\homeo_+(\T)$ with a global fixed point is always semi-conjugate to a subgroup of $\psl(2,\RR)$. We extend this result to \emph{elementary} subgroups, that is, subgroups of $\homeo_+(\T)$ preserving a Borel probability measure on $\T$.  Concretely, a subgroup $G\le \homeo_+(\T)$ is elementary if either it admits a finite orbit or it is (continuously) semi-conjugate to a subgroup of rotations (see the discussion in \S \ref{ssc.elementary}, or Ghys \cite[Proposition 6.17]{ghys-circle}).
\begin{remark}
		This terminology is coherent with the usual notion of elementary subgroup of $\psl(2,\RR)$, which is a subgroup whose action on $\overline {\mathbb D}$ has a finite orbit. Concretely, an elementary subgroup of $\psl(2,\RR)$ is either conjugate in $\psl(2,\RR)$ to a subgroup of rotations, or has a finite orbit on the circle at infinity (which has to be either a fixed point or a periodic orbit of order 2); see Kim, Koberda, and Mj \cite[Proposition 2.1]{Kim-Koberda-Mj}.
\end{remark}

In practice, compared to the work of Kova\v{c}evi\'c, we also consider the case of periodic orbit of order 2, which is a bit harder than the other cases.

\begin{theoremA}\label{tA.elementary_mobius-Like}
	If $G< \homeo_+(\T)$ is an elementary, M\"obius-like subgroup, then $G$ is continuously semi-conjugate to an elementary subgroup of $\psl(2,\RR)$ and, moreover, the corresponding morphism $G\to \psl(2,\RR)$ is injective.
\end{theoremA}

Our proof is a modern revised version of that from \cite{Ko1}, which is somehow difficult to digest as Kova\v{c}evi\'{c} needs to prove Solodov's theorem along way. Originally presented in \cite[Theorem 3.21]{Solodov}, this is now a classical result in the subject (see Ghys \cite[Theorem 6.12]{ghys-circle} and Farb--Franks \cite[Theorems 1.3 and 1.4]{Farb-Franks}). For further reference, it is more convenient to add a first item to its usual statement, which is a direct consequence of H\"older's theorem \cite[Theorem 6.10]{ghys-circle}.
\begin{theorem}[Solodov]\label{t.solodov} Let $G< \homeo_+(\mathbb{R})$ be a subgroup with at most $1$ fixed point.  Then 
	\begin{itemize}
		\item either the action of $G$ admits a unique fixed point $p$, in which case the restriction of the action of $G$ to $(-\infty,p)$ (respectively, to $(p,+\infty)$) is continuously semi-conjugate to an action by translations,  and the corresponding morphism $G\to \RR$ is injective,  or
		\item the action of $G$ is continuously semi-conjugate to an action by affine transformations, and the corresponding morphism $G\to \aff_+(\RR)$ is injective.
	\end{itemize}
\end{theorem}
In the opposite direction, Kova\v{c}evi\'{c} \cite{Ko2} showed that the above result fails for general M\"obius-like groups.

\begin{theorem}[Kova\v{c}evi\'{c}]
	There exist finitely generated (even finitely presented) M\"obius-like subgroups of $\homeo_+(\T)$ whose action is minimal (every orbit is dense) but not conjugate into $\psl(2,\RR)$.
\end{theorem}

When a subgroup of $\homeo_+(\T)$ has at most 2 fixed points but is not semi-conjugate into $\psl(2,\RR)$ (or not even Möbius-like), one may still wonder whether it is \emph{abstractly} isomorphic to a subgroup of $\psl(2,\RR)$. Building on the proof of Theorem \ref{tA.elementary_mobius-Like}, we settle this problem for the negative. 

\begin{theoremA}\label{t.smooth_example_non_isomorphic}
	There exists a finitely generated, elementary group of smooth ($C^\infty$) circle diffeomorphisms, with at most 2 fixed points, and which is not isomorphic to any subgroup of $\mathrm{PSL}(2,\mathbb{R})$.
\end{theoremA}

For the next main result, we will consider the following topological condition (inspired by work of Rebelo \cite{rebelo} in the real-analytic setting). To state it, recall (see \cite[Proposition 5.6]{ghys-circle}) that any subgroup $G\le \homeo_+(\T)$ admits a non-empty $G$-invariant compact subset $\Lambda\subset \T$ which is minimal with respect to inclusion, called a \emph{minimal invariant subset}; moreover, if $G$ does not have any periodic orbit, the minimal invariant subset is always unique, and it is either the whole circle or a Cantor set. In particular, when $G\le \homeo_+(\T)$ is a non-elementary subgroup, we can refer to \emph{the} minimal invariant subset of $G$. We also say that  a non-empty interval $I\subset \T$ is \emph{wandering} (for $G$) if $\Lambda\cap I=\varnothing$.

\begin{definition}\label{d.local_discrete}
	A non-elementary subgroup $G\le \mathrm{Homeo}_+(\mathbb{S}^1)$ of circle homeomorphisms is \emph{locally discrete} if for every non-wandering interval $I\subset \T$, the identity is isolated among the collection of restrictions $\{g|_I:g\in G\}\subset C^0(I;\T)$, with respect to the $C^0$ topology.
\end{definition}

It is clear that if a subgroup is locally discrete, then it is discrete in the usual sense. In fact, we will see that in the $C^0$ topology, the two properties are equivalent (Lemma \ref{l.non_loc_discrete_entao_non_discrete}), whereas in the real-analytic setting of Rebelo \cite{rebelo}, the equivalence is an open problem (see the discussion in Alvarez \textit{et al.}\ \cite[Remark 1.9]{PingPong2}).
The condition of considering non-wandering intervals only in the definition sounds technical, but it is actually the appropriate one (see Proposition \ref{l.non-elementary_loc_non-discrete}).  For non-locally discrete groups with at most $N$ fixed points, we prove the following.

\begin{theoremA}\label{mainthmA.non-discrete}
	Let $G<\homeo_+(\T)$ be a non-elementary, non-locally discrete subgroup with at most $N$ fixed points. Then, there exists $k\ge 1$ such that $G$ is conjugate to a dense subgroup of $\psl^{(k)}(2,\RR)$.
\end{theoremA}

\begin{remark}
	The degree of the extension $k$ is given by half of the largest number of fixed points for non-trivial elements in $G$ (which is thus necessarily even).
\end{remark}

\begin{corA}
	Let $G<\homeo_+(\T)$ be a non-elementary, non-locally discrete subgroup with at most $2$ fixed points. Then, $G$ is conjugate to a dense subgroup of $\psl(2,\RR)$.
\end{corA}

\begin{remark}
	Another consequence of Theorem \ref{mainthmA.non-discrete} is that, if moreover the group $G$ is finitely generated, it necessarily contains elements with irrational number (see Kim, Koberda, and Mj \cite[Lemma 2.24]{Kim-Koberda-Mj}). Actually, a key step of the proof (see Lemma \ref{l.contains_irrational}) is to show that the \emph{closure} $\overline G$ contains an element with irrational rotation number. Let us also comment that when considering finitely generated subgroups $G\le \mathrm{Diff}^1_+(\T)$ which are non-locally discrete with respect to the \emph{$C^1$ topology}, the analogue conclusion is still open, even in real-analytic regularity (see Matsuda \cite{Matsuda2009}).
\end{remark}

Let us describe the structure of the paper. After introducing the main terminology in Section \ref{s.preliminaries}, we discuss Theorems \ref{tA.elementary_mobius-Like} and \ref{t.smooth_example_non_isomorphic} in Section \ref{sc.elementary}. The last section is devoted to the proof of Theorem \ref{mainthmA.non-discrete}. Let us give a quick outline of the proof. First of all, we remark that the action of $G$ is minimal (Lemma \ref{l.non_loc_discrete_entao_non_discrete}). The next step is to understand which elements belong to the $C^0$ closure $\overline G$: we prove that $\overline{G}$ has also at most $N$ fixed points (Lemma \ref{l.locally_globally}) and, up to conjugacy, it contains the group of rotations $\SO(2)$ (Lemma \ref{l.contains_irrational}). We then conclude by using the classification by Giblin and Markovic \cite{GiblinMarkovic} (see Theorem \ref{t.giblin-markovic}) of closed subgroups of $\homeo_+(\T)$ acting transitively and containing a non-constant continuous path.

\section{Preliminaries}\label{s.preliminaries}

\subsection{Topology of the circle and of the group of circle homeomorphisms} In what follows, we let $\homeo_+(\T)$ be the group of orientation-preserving homeomorphisms of the circle. Here, the circle $\T$ will be considered as the one-dimensional torus $\RR/\ZZ$, with its normalized Lebesgue measure. Given a Borel subset $A\subset \T$, we denote by $|A|$ its Lebesgue measure. Any $f\in \homeo_+(\T)$ lifts to a homeomorphism $F:\RR\to \RR$ commuting with integer translations, and this lift is defined uniquely, up to integer translations. This allows to identify the universal cover of $\homeo_+(\T)$ with the group $\homeo_{\ZZ}(\RR)$ of homeomorphisms of the line commuting with integer translations.
If $x,y\in \T$ are two distinct points, we denote by $(x,y)$ the (open) interval of points $z\in \T$ such that the triple $(x,z,y)$ is positively ordered. When $x=y$, $(x,x)$ is by convention, just the empty set. Similarly we denote other kind of intervals $[x,y]$, $[x,y)$, $(x,y]$. We can then define the distance $d(x,y)=\min \{|(x,y)|,|(y,x)|\}$.  When $d(x,y)<1/2$, this is the usual distance between points in any Euclidean chart containing $x,y$, so that we can write (with slight abuse of notation) $|x-y|$ instead of $d(x,y)$.
We want to consider the $C^0$ topology on $\homeo_+(\T)$, and in order to quantify this, we consider the distance $d_{\infty}(f,g):=\sup_{x\in \T}d(f(x),g(x))$. It is well-known that the distance $d_\infty$ on $\homeo_+(\T)$ is not complete, and for this reason we will also consider the additional distance $d_{\mathcal C^0}(f,g):=d_\infty(f,g)+d_\infty(f^{-1},g^{-1})$, which makes $\homeo_+(\T)$ complete. When $d_\infty(f,g)<1/2$, we can take any lifts $\widetilde{f},\widetilde{g}\in \homeo_{\ZZ}(\RR)$ of $f$ and $g$, respectively, such that $|\widetilde{f}(0)-\widetilde{g}(0)|<1$, and we have $d_\infty(f,g)=\|\widetilde{f}-\widetilde{g}\|$, where $\|\cdot\|$ is the usual uniform norm on continuous functions. Because of this, when $d_\infty(f,g)<1/2$, we will also write $\|f-g\|$ for the distance $d_\infty(f,g)$, and tacitly make computations in a local Euclidean chart.
This also justifies the following definition.
\begin{definition}\label{d.close-to-identity-positive}
	Let $f\in\homeo_+(\T)$ be a circle homeomorphism  such that $d_\infty(f,\idid)<1/2$. We say that $f$ is \emph{positive} if for every $x\in\T$, we have $f(x)\in(x,x+\frac{1}{2})$. When $f^{-1}$ is positive, we say that $f$ is \emph{negative}.
	Given an interval $I\subset \T$, we say that $f$ is positive (respectively, negative) on $I$ if the previous conditions hold simply for any $x\in I$.
\end{definition}

\begin{remark}\label{rem.composition_positive}
	Observe that the definition above implies that a positive (respectively, negative) homeomorphism does not have any fixed point. Moreover, when $f,g\in \homeo_+(\T)$ are positive, the composition $fg$ is also positive, provided that $d_\infty(f,\idid)+d_\infty(g,\idid)<1/2$.
\end{remark}

As we work with homeomorphisms having a prescribed bound on the number of fixed points, it is fundamental to have in mind that if $f,g\in \homeo_+(\T)$ and $x\in \T$ are such that $f(x)=g(x)$, then $g^{-1}f(x)=x$. When this occurs, we say that $f$ and $g$ \emph{cross}. More specifically, $I\subset \T$, we say that $f$ and $g$ cross \emph{in $I$} if there exists $x\in I$ such that $f(x)=g(x)$. More specifically, we say that $f$ and $g$ cross \emph{hyperbolically} at $x$, if $x$ is a (topologically) hyperbolic fixed point (either attracting, or repelling) for the composition $fg^{-1}$: equivalently, one has \[(f(x-\eta)-g(x-\eta))(f(x+\eta)-g(x+\eta))<0\] for any sufficiently small $\eta>0$.  
We will also say that $f$ and $g$ cross \emph{weakly hyperbolically} in an interval $I$ if there exists an interval $J=[x,y]\subset I$ (possibly, $x=y$) such that $f\vert_J=g\vert_J$ and 
\[(f(x-\eta)-g(x-\eta))(f(y+\eta)-g(y+\eta))<0\]
for any sufficiently small $\eta>0$.
\begin{remark}\label{rem:whyp_stable}
	Having a weakly hyperbolically crossing in some open interval $I$ is stable under sufficiently small $C^0$ perturbations.
\end{remark}

We will be interested in counting the number of points for which $f(x)=g(x)$, so that we will say that $f$ and $g$ cross once, twice, 3 times, and so on.

\subsection{Semi-conjugacy of group actions on the circle}

A group action on a topological space $X$ is semi-conjugate to an action of the group on a topological space $Y$ if there exists a continuous equivariant map from $X$ to $Y$.
This classical notion of semi-conjugacy in dynamical systems can be slightly modified to get an equivalence relation on group actions on the line and the circle. For a detailed discussion, we refer to the monograph by Kim, Koberda, and Mj \cite{Kim-Koberda-Mj}, from which we borrow some terminology.

Let $G$ be a group, and $\varphi,\psi:G\to \homeo_+(\RR)$ two homomorphisms. We say that $\varphi$ and $\psi$ are \emph{conjugate} if there exists $h\in \homeo_+(\RR)$ which is \emph{$(\varphi,\psi)$-equivariant}: $h(\varphi(g)(x))=\psi(g)(h(x))$ for any $g\in G$ and $x\in \RR$. We say that $\varphi$ and $\psi$ are \emph{semi-conjugate} if there exists a non-decreasing map $h:\RR\to \RR$ such that $|h(x)|\to \infty$ as $|x|\to \infty$, and which is $(\varphi,\psi)$-equivariant. When the map $h$ is continuous, we recover the usual notion of continuous semi-conjugacy from $\varphi$ to $\psi$.

Similarly, when $\varphi,\psi:G\to \homeo_+(\T)$, we say that $\varphi$ and $\psi$ are conjugate if there exists a $(\varphi,\psi)$-equivariant circle homeomorphism. However, for extending the notion of semi-conjugacy to this situation, we have to pass through actions on the line: using the cyclic central extension
\[
1\to \ZZ\to \homeo_{\ZZ}(\RR)\to \homeo_+(\T)\to 1,
\]
we say that $\varphi,\psi:G\to \homeo_+(\T)$ are semi-conjugate if we can find a cyclic central extension
\[
1\to \ZZ \to \widetilde{G} \to G\to 1,
\]
and semi-conjugate homomorphisms $\widetilde{\varphi},\widetilde{\psi}:\widetilde{G}\to \homeo_{\ZZ}(\RR)$ (in the previous sense), which send any integer $n\in \ZZ\subset \widetilde G$ to the translation by $n$, and lift the homomorphisms $\varphi$ and $\psi$, respectively: in short, the following diagram must commute
\[ \begin{tikzcd}
	& & \widetilde{G} \arrow{r} \arrow{d}{\widetilde{\varphi},\widetilde{\psi}} & G \arrow[swap]{d}{\varphi,\psi} \arrow{rd} \\
	1 \arrow{r} &\ZZ \arrow{r} \arrow{ur} & \homeo_{\ZZ}(\RR) \arrow{r} & \homeo_+(\T) \arrow{r} & 1
\end{tikzcd}
\]
When $\widetilde\varphi$ is continuously semi-conjugate to $\widetilde\psi$, we say that $\varphi$ is continuously semi-conjugate to $\psi$. Note that this does not coincide with usual notion in dynamical systems, as we only allow continuous semi-conjugacies through degree-1 maps of the circle.

With abuse of terminology, we will also say that the subgroups $\varphi(G)$, $\psi(G)$ are (continuously) \mbox{(semi-)}conjugate, and more generally, we say that two subgroups of homeomorphisms $G_1$ and $G_2$ are (continuously) \mbox{(semi-)}conjugate, if there exist a group $G$ and surjective homomorphisms $\varphi_i:G\to G_i$ (for $i\in \{1,2\}$) which are (continuously)  \mbox{(semi-)}conjugate.

\section{Elementary groups}\label{sc.elementary}

\subsection{Basic results}\label{ssc.elementary}
Recall from the introduction that a subgroup $G$ of $\homeo_+(\T)$ is \emph{elementary} if its action preserves a Borel probability measure.
Recall from the introduction that an elementary subgroup $G< \homeo_+(\T)$, either admits a finite orbit or  is semi-conjugate to a subgroup of rotations.
We assume that the reader is familiar with the notion of rotation number for circle homeomorphisms. It is well-known that given a circle homeomorphism $g\in \homeo_+(\T)$, we can compute its rotation number $\rot(g)\in \RR/\ZZ\cong\T$ by taking \emph{any} $g$-invariant Borel probability measure $\mu$ and point $x\in \T$, and measure how much $x$ is displaced: $\rot(g)=\mu[x,g(x))$. Therefore, if $G< \homeo_+(\T)$ is an elementary subgroup, with invariant Borel probability measure $\mu$, then one has
\[
\rot(fg)=\mu[x,fg(x))=\mu[x,g(x))+\mu[g(x),fg(x))=\rot(g)+\rot(f),
\]
which means that the restriction to $G$ of the function rotation number defines a homomorphism. More precisely, it gives a homomorphism $\rho_G:G\to \SO(2)$ defined by $g\mapsto R_{\rot(g)}$ (which does not depend on the choice of the invariant measure $\mu$), and the actions of $G$ and $\rho_G(G)$ are semi-conjugate (according to the terminology introduced in the previous section). Moreover, if the action of $G$ has no finite orbit, then $G$ is continuously semi-conjugate to $\rho_G(G)$ (see Ghys \cite[Proposition 6.17]{ghys-circle}).
We have the following basic result.

\begin{lemma}\label{l.kernel_fixes_measure}
	Let $G< \homeo_+(\T)$ be a subgroup with an invariant Borel probability measure $\mu$ on $\T$. Then, the kernel of the homomorphism $\rho_G:G\to \SO(2)$ fixes $\supp(\mu)$ pointwise.
\end{lemma}
\begin{proof}
	Take $x\in \supp(\mu)$. If $g\in G$ is such that $g(x)\ne x$, then $\mu[x,g(x))\neq 0$ and thus $\rot(g)\neq 0$. 
\end{proof}

Let us give a straightforward application in the case when the elementary subgroup has at most $N$ fixed points.

\begin{lemma}\label{l.atomless_measure}
	Let $G< \homeo_+(\T)$ be a subgroup with at most $N$ fixed points, preserving a Borel probability measure $\mu$ on $\T$, whose support contains at least $N+1$ points. Then, the homomorphism $\rho_G:G\to \SO(2)$ is injective, and $G$ is continuously semi-conjugate to $\rho_G(G)$.
	
	In particular, $G$ is isomorphic to a subgroup of $\mathrm{SO}(2)$, and thus abelian. 
\end{lemma}

\begin{proof}
	If there was an element in the kernel, by Lemma \ref{l.kernel_fixes_measure}, it would fix the support of $\mu$, and thus at least $N+1$ points. As we are assuming that $G$ acts with at most $N$ fixed points, this gives that the kernel is trivial. Moreover, if the action of $G$ has a finite orbit, then the image $\rho_G(G)$ is finite, and so is $G$. In this case, we have that $G$ is conjugate to $\rho_G(G)$ (see e.g. Herman \cite[\S II.6]{Herman}). If the action of $G$ has not finite orbit, then by \cite[Proposition 6.17]{ghys-circle} we get that $G$ is continuously semi-conjugate to $\rho_G(G)$.
\end{proof}

\subsection{Möbius-like elementary groups}

Here we discuss the first result, namely Theorem \ref{tA.elementary_mobius-Like}, which states that elementary, M\"obius-like  subgroups of $\homeo_+(\T)$ are continuously semi-conjugate to subgroups of $\psl(2,\RR)$.

\begin{proof}[Proof of Theorem \ref{tA.elementary_mobius-Like}]
	Let $\nu$ be a Borel probability measure preserved by the action of $G$ on the circle. If $\supp(\nu)$ contains at least 3 points, then Lemma  \ref{l.atomless_measure} gives that $G$ is continuously semi-conjugate to the subgroup of rotations $\rho_G(G)\le \SO(2)< \psl(2,\RR)$, and actually isomorphic.
	
	Thus, from now on, we can assume that the action of $G$ has a finite orbit, which is either a fixed point or a pair of points. 	
	Assume first that there is a unique global fixed point for $G$. Then, by Solodov's theorem (Theorem \ref{t.solodov}), $G$ is continuously semi-conjugate to a subgroup of affine transformations, and the corresponding morphism $G\to \aff_+(\RR)<\psl(2,\RR)$ is injective.
	Assume next that $G$ has two global fixed points $p,q\in \T$.
 	After H\"older's theorem (Theorem \ref{t.solodov}), the restrictions of the actions of $G$ to the two connected components of $\T\setminus \{p,q\}$ are both continuously semi-conjugate to actions by translations, and moreover the corresponding induced homomorphisms $G\to \RR$ and $G\to \RR$ are injective. Therefore, after a continuous global semi-conjugacy (sending $p$ and $q$ to the points $0$ and $\infty$ of $\RR \mathrm P^1=\RR\cup \{\infty\}$), we can identify $G$ with a M\"obius-like subgroup $\Gamma$ of the group $\Lambda\cong \RR^2$ of homeomorphisms $f:\RR \mathrm P^1\to \RR \mathrm P^1$ for which there exist $\alpha=\alpha(f),\beta=\beta(f)\in \RR$ such that
 	\[f(x)=\left\{\begin{array}{lr} e^\alpha x & \text{ for } x\in [0,\infty] , \\
 		e^\beta x & \text{ for } x\in [\infty, 0 ]. \end{array} \right . \]
 	From now on, we will work with $\Gamma$ instead of $G$. Take a non-trivial element $f\in \Gamma$, and replacing $f$ by $f^{-1}$ if necessary, we assume $\alpha(f),\beta(f)>0$.
 	Up to replace $\Gamma$ by the subgroup of $\Lambda$ obtained by conjugating $\Gamma$ by the homeomorphism
 	\[
 	h(x)=\left\{\begin{array}{lr} x^{1/\alpha(f)} & \text{ for } x\in [0,\infty] , \\
 		-(-x)^{1/\beta(f)} & \text{ for } x\in [\infty, 0 ], \end{array} \right . \]
	we can assume that $\alpha(f)=\beta(f)=1$.
	
	We claim that also for every other $g\in \Gamma$, we have $\alpha(g)=\beta(g)$.
	Indeed, if for some $g\in \Gamma$, $\alpha=\alpha(g)$ and $\beta=\beta(g)$ satisfy $\alpha-\beta\neq0$, then there exists an integer $N\in \ZZ$ such that $N(\alpha-\beta) >1$ and so, there exists a second integer $M\in\ZZ$ such that
	$N \alpha > M > N \beta$.
	Therefore, $N \alpha - M > 0$ and $N \beta-M < 0$ and so for the element $\gamma=g^Nf^{-M}\in \Gamma$ one has $\alpha(\gamma)>0$ and $\beta(\gamma)<0$. This contradicts the assumption that $\Gamma$ is M\"obius-like. Now, since the subgroup of elements $g\in \Lambda$ with $\alpha(g)=\beta(g)$ is a subgroup of $\psl(2,\RR)$ (it is the stabilizer of the points $0$ and $\infty$), we conclude that $\Gamma$ is in $\psl(2,\RR)$.

	For the last case, we will assume that $G$ has a finite orbit of order $2$, and denote by $\nu$ the corresponding invariant probability measure. After Lemma \ref{l.kernel_fixes_measure}, we have a short exact sequence
	\begin{equation*}\label{e.exact_sequence_theorem_2.18}
		1\rightarrow G_0 \rightarrow G \rightarrow \ZZ_2 \rightarrow 1,
	\end{equation*}
	where we write $G_0=\ker(\rho_G)$. By the previous case (two global fixed points), $G_0$ is isomorphic and continuously semi-conjugate to a subgroup of the stabilizer of $0$ and $\infty$ in $\psl(2,\RR)$, which is isomorphic to $\RR$. By continuity of the rotation number, the closure $\overline G$ in $\homeo_+(\T)$ also fits a short exact sequence
	\begin{equation*}\label{e.exact_sequence_theorem_2.18}
		1\rightarrow \overline{G_0} \rightarrow \overline{G} \rightarrow \ZZ_2 \rightarrow 1,
	\end{equation*}
	and acts on the circle fixing the same two fixed points as $G_0$. In particular, it is left-orderable, and actually	order-isomorphic to either $\ZZ$ or $\RR$, with the standard order (or the reflected one).
	Let $a\in G$ be an element of the group with $\rot(a)= 1/2$.
	Observe that if $a^2\neq \idid$ then $a^2$ fixes the $2$ atoms of $\nu$ and no further point, so they are both parabolic fixed points. This contradicts the M\"obius-like assumption. Thus, we must have $a^2=\idid$. Therefore the exact sequence splits, so that $\overline{G}\cong \overline{G_0} \rtimes_A \ZZ_2$, where $A$ is the involution of $\mathrm{Aut}(\overline{G_0})$ defined by the conjugacy by $a$ (clearly this also gives the isomorphism
		$G\cong G_0 \rtimes_A \ZZ_2$, where we keep denoting by $A$ the restriction of the involution to $G_0$). We denote by $-\idid\in \mathrm{Aut}(\overline {G_0})$ the involution that sends any element to its inverse.
			
			\begin{claim}\label{claim:elementary}
					We have $A=-\mathrm{id}$, thus $G\cong G_0  \rtimes_{-\mathrm{id}} \ZZ_2$.
				\end{claim}
			
			\setcounter{claim}{0}
			
			\begin{proof}[Proof of claim]
		
		Note that $A$ acts as an order-automorphism of order 2 on $\overline{G_0}$, which is isomorphic to either $\ZZ$ or $\RR$. Hence it must be $A\in \{\idid,-\idid\}$. However, if $A=\idid$, then any non-trivial element of $G_0$ is centralized by $a$, and therefore its two fixed points should have the same dynamical nature (both parabolic), contradicting the Möbius-like assumption. We conclude that $aga^{-1}=g^{-1}$ for any element of $\overline{G_0}$.
	\end{proof}
	
	Let us now see that $G$ is continuously semi-conjugate to a subgroup of $\psl(2,\RR)$. Let $p,q=a(p)\in \T$ be the finite orbit, and take a $G_0$-invariant atomless Radon measure $\mu$ on $(p,q)$ (coming from H\"older's theorem). When $y<x$, we will use the convention $\mu[x,y)=-\mu(y,x]$. With such notation, we have that for any $f\in G_0$, there exists a constant $\alpha(f)\in \RR$ such that $\mu[x_0,f(x))=\mu[x_0,x)+\alpha(f)$ for any choice of $x_0,x\in (p,q)$. The function $\alpha:G_0\to \RR$ is an injective homomorphism, and it actually corresponds, up to a scalar multiple, to the $\alpha(f)$ previously considered.
	Now, fix a point $x_0\in (p,q)$, and consider the function $h:\T\to \RR \mathrm P^1$ defined by
	\[
	h(x)=\left\{
	\begin{array}{lr}
		0 & \text{if }x=p,\\
		e^{\mu[x_0,x)} & \text{if }x\in (p,q),\\
		\infty & \text{if }x=q,\\
		-e^{-\mu[x_0,a(x))} & \text{if }x\in (q,p).
	\end{array}
	\right.
	\]
	The function $h$ is continuous and of degree 1. Let us check that it defines a semi-conjugacy from $G$ to the subgroup of $\psl(2,\RR)$, generated by the maps
	\[
	A_f(t)=e^{\alpha(f)}t \quad\text{(for $f\in G_0$)},
	\]
	and $I(t)=-1/t$. Indeed, if $x\in (p,q)$, then we have
	\[
	h(a(x))=-e^{-\mu[x_0,a(a(x)))}=-1/e^{\mu[x_0,x)}=-1/h(x)=I(h(x)),
	\]
	and for $f\in G_0$,
	\[
	h(f(x))=e^{\mu[x_0,f(x))}=e^{\mu[x_0,x)+\alpha(f)}=e^{\alpha(f)}h(x)=A_f(h(x)).
	\]
	When $x\in (q,p)$, then we have
	\[
	h(a(x))=-e^{\mu[x_0,a(x))}=-1/(-e^{-\mu[x_0,x)})=-1/h(x)=I(h(x)),
	\]
	and for $f\in G_0$,
	\[
	h(f(x))=-e^{-\mu[x_0,af(x))}=-e^{-\mu[x_0,f^{-1}a(x))}=
	-e^{-\mu[x_0,a(x))-\alpha(f^{-1})}=
	e^{\alpha(f)}h(x)=A_f(h(x)).
	\]
	This concludes the proof.
\end{proof}

\subsection{A group of smooth diffeomorphisms which is not isomorphic to any M\"obius group}

Here we give an example of group of $C^\infty$ circle diffeomorphisms, with at most 2 fixed points, but which is not isomorphic to any subgroup of  $\mathrm{PSL}(2,\mathbb{R})$.

\begin{proof}[Proof of Theorem \ref{t.smooth_example_non_isomorphic}] For the construction of this example we will consider the following maps with respect to the projective coordinates of the circle: fix $\lambda,\mu>1$ such that $\log \lambda$ and $\log\mu$ are linearly independent over $\QQ$, and set 
	$$f(x)=\left\{\begin{array}{lr} \lambda x & \text{ for } x\in [0,\infty] , \\
		\mu x & \text{ for } x\in [\infty, 0 ]. \end{array} \right . $$
	For convenience, set $g = R_{\frac{1}{2}} f R_{\frac{1}{2}}$. It is clear that $f$ and $g$ generate a rank $2$ abelian free group. Moreover, conjugation by the rotation $R_{\frac{1}{2}}$ defines an action of $\ZZ_2$ on such $\ZZ^2$ given by the matrix 
	$$A = \begin{pmatrix}
		0 & 1 \\
		1 & 0 
	\end{pmatrix} 
	$$
	(with respect to the basis $f, g$). In other terms, the group $G= \langle f, R_{\frac{1}{2}} \rangle $ is isomorphic to the semi-direct product $\ZZ^2 \rtimes_{A} \ZZ_2$. By construction, one can observe that $G$ acts with at most $2$ fixed points. 
	
	\begin{claim*}
		The group $G$ is not isomorphic to any subgroup of $\psl(2,\RR)$.
	\end{claim*}
	\begin{proof}[Proof of claim]
	By contradiction, let us assume $G$ is isomorphic to some subgroup $H\le \psl(2,\RR)$.
	Then $H\cong \ZZ^2 \rtimes_{A} \ZZ_2$ is virtually abelian, and therefore its action on $\T$ preserves a Borel probability measure. We conclude that $H$ is elementary and M\"obius-like. In particular, it satisfies the assumptions of Theorem \ref{tA.elementary_mobius-Like}. Let us go through the cases discussed in its proof. If $H$ preserves an atomless Borel probability measure, that it must be abelian by Lemma \ref{l.atomless_measure}, but this is a contradiction. In the case $H$ has fixed points, then by Solodov's theorem, $H$ is isomorphic to a subgroup of $\aff_+(\RR)$, and therefore torsion free. This is again a contradiction. Therefore we are left with the case of order-2 orbit, so by Claim \ref{claim:elementary}, we have 
		\begin{equation}\label{eq:semid1}
		H\cong \ZZ^2\rtimes_{-\mathrm{id}}\ZZ_2=\langle f,g,a\mid [f,g]=\idid,\,a^2=\idid,\,afa^{-1}=f^{-1},\,aga^{-1}=g^{-1}\rangle.
		\end{equation}
	However, this contradicts the fact that
		\begin{equation}\label{eq:semid2}
		H\cong\ZZ^2\rtimes_{A}\ZZ_2=\langle f,g,a\mid [f,g]=\idid,\,a^2=\idid,\,afa^{-1}=g\rangle,
		\end{equation}
	because the the semi-direct products at lines \eqref{eq:semid1} and \eqref{eq:semid2} have different abelianizations (as one easily checks from the presentations: \eqref{eq:semid1} gives $(\ZZ_2)^3$, while \eqref{eq:semid2} gives $\ZZ\times \ZZ_2$). This ends the proof of the claim.
\end{proof}
	
	Finally, conjugating $G$ by a suitable $C^{\infty}$ homeomorphism which is infinitely flat at $0$ and $\infty$, we can embed $G$ into $\mathrm{Diff}^{\infty}(\T)$.
	This proves Theorem~\ref{t.smooth_example_non_isomorphic}.
\end{proof}

\begin{remark}\label{r.nao_isomorfo_minimal}
	An example of a \emph{minimal} finitely generated group of circle homeomorphisms, with at most 2 fixed points, and which is not isomorphic to any subgroup of $\mathrm{PSL}(2,\mathbb{R})$, can be built by taking a free product of the group $G$ from Theorem \ref{t.smooth_example_non_isomorphic} with itself, obtained by blowing-up a free orbit. This kind of construction is inspired by \kovacevic's work \cite{Ko2}, and it is detailed in the second author's PhD thesis: indeed, let $G_1$ and $G_2$ be two copies of the group $G$, considered as acting on distinct circles $\Gamma_1$ and $\Gamma_2$, and choose any two points $x\in \Gamma_1$ and $y\in \Gamma_2$ with trivial stabilizer in $G_1$ and $G_2$, respectively. Then, \cite[Theorem D]{theseJoao} gives that the \emph{amalgamated product} $(G_1,x)\star (G_2,y)$ (properly defined after \cite[Theorem 4.7]{theseJoao}; note that in this case it is simply a \emph{free} product) is a subgroup of $\homeo_+(\T)$ acting minimally, and with at most 2 fixed points.	
\end{remark}

\section{Non-locally discrete groups with at most $N$ fixed points}\label{s.non-discrete}
In this section we discuss Theorem \ref{mainthmA.non-discrete}, about non-locally discrete groups with at most $N$ fixed points. We start this section with an example that justifies our definition of non-local discreteness (Definition \ref{d.local_discrete}).

We first need a result that can be seen as a consequence of the ``Projective Baumslag Lemma'' \cite[Lemma 3.4]{Kim-Koberda-Mj}.
\begin{lemma}\label{lem:KKM}
	Let $A\le \psl(2,\RR)$ be a countable subgroup. Then, there exists a countable set $D\subset \SO(2)$ such that for every rotation $R_\rho\in \SO(2)\setminus D$, the subgroup $\langle A,R_\rho\rangle$ is isomorphic to the free product $A*\langle R_\rho\rangle$.
\end{lemma}
\begin{proof}
	The subgroup $\SO(2)$ is a maximal abelian subgroup of $\psl(2,\RR)$, acting without fixed points. Let
	\[w=(k,g_1,\ldots,g_k,m_1,\ldots,m_k)\]
	be a choice of an integer $k\ge 1$, non-trivial elements $g_1,\ldots,g_k \in A$, and finitely many integers $m_1,\ldots,m_k\in \ZZ\setminus \{0\}$ (so, there are countably many such choices).	
	By \cite[Lemma 3.4]{Kim-Koberda-Mj}, for any such $w$, there exists a finite subset $D_w\subset \SO(2)$ such that for any $R_\rho\in \SO(2)\setminus D_w$ the composition
	\[
	\phi(w)=g_1R_\rho^{m_1}\cdots g_k R_\rho^{m_k}
	\]
	is non-trivial (more precisely, the square of its trace is different from $4$).
	We can then take $D=\bigcup_{w}D_w$, which is a countable union of finite subsets, and hence countable.
\end{proof}

\begin{proposition}\label{l.non-elementary_loc_non-discrete}
	There exists a finitely generated non-elementary subgroup $G<\homeo_+(\T)$ with at most 2 fixed points, which is not M\"obius-like, and for which there exists a wandering interval $I\subset\T$ such that the image of $\stab(G,I)$ is non-discrete in $\homeo_+(I)$.
\end{proposition}
\begin{proof}
	First, we will construct a non-elementary subgroup of $\psl(2,\RR)$ with the stabilizer of a point $p\in\T$ being parabolic and isomorphic to $\ZZ^2$. 
	For this, let $T_\alpha, T_\beta$ be two parabolic elements of $\psl(2,\RR)$ fixing the same point $p\in\T$, such that the subgroup $T=\langle T_\alpha, T_\beta\rangle$ is free abelian of rank 2. After Lemma \ref{lem:KKM}, we can find an irrational rotation $R_\rho\in \SO(2)$ such that 	$F=\langle T,R_\rho\rangle$ is isomorphic to the free product $T*\langle R_\rho\rangle\cong \ZZ^2*\ZZ$. Observe that the stabilizer of the point $p$ has not changed, that is $\stab(F,p) = \stab(T,p) = T$. Indeed, for every element $g\in\psl(2,\RR)$ that fixes the point $p$, we have that the parabolic element $gT_\alpha g^{-1}$ fixes $p$ and so it commutes with $T_\alpha$; so if $\stab(F,p) \neq T$, we could then find an element $g\in F\setminus T$ such that $[gT_\alpha g^{-1},T_\alpha]=\idid$. This is not possible after our choice of $R_\rho$.
	Finally, note that $F$ is a non-elementary subgroup of $\psl(2,\RR)$ with a parabolic stabilizer $T$ of the point $p$ isomorphic to $\ZZ^2$.
	
	For the second step of the construction, we \emph{blow-up} the action of $F$ at the orbit of $p$ (see for instance Kim and Koberda \cite{Kim-Koberda-Denjoy}), in such a way that the action of the stabilizer $T$ on the interval $I$ inserted at $p$ is conjugate to the original action of $T$ on $\T\setminus \{p\}\cong \RR$ (which is a minimal action by translations). Choosing the good orientation for this $T$-action on $I$, we have that any non-trivial element of $T$ acts on the new circle with 2 parabolic fixed points. We call $G$ the resulting subgroup of $\homeo_+(\T)$, which is abstractly isomorphic to $F\cong \ZZ^2*\ZZ$. 
	It is not difficult to verify that any non-trivial element of $G$ has at most $2$ fixed points (we do not have changed the action on the complement of the $F$-orbit of $p$, and conjugates of $T$ are now acting with at most 2 fixed points).
	Therefore, $G$ is a non-elementary group of circle homeomorphisms, with at most $2$ fixed points. The subgroup $\stab(G,I)\cong T$ acts on $I$ minimally by translations, and thus its image in $\homeo_+(I)$ is non-discrete.
\end{proof}

It turns out that non-discreteness and non-local discreteness are equivalent for non-elementary subgroups. To see this, we need a fundamental structural result for non-elementary subgroups. To state it, we say that the action of a subgroup $G\le \homeo_+(\T)$ is \emph{proximal} if for every non-empty open intervals $I,J\subset \T$, there exists an element $g\in G$ such that $g(I)\subset J$. Note that if the action is proximal, then it is automatically minimal, and $G$ non-elementary.
Being proximal is not invariant under semi-conjugacy, since if $J$ is a wandering interval and $I$ is not, there is no way we can send $I$ inside $J$. For this reason,
we also say that the action of a non-elementary subgroup $G$ is proximal \emph{in restriction to the minimal invariant subset} if the previous statement holds only for intervals $J\subset \T$ which are non-wandering (that is, intersecting the minimal invariant subset). This property is preserved under semi-conjugacy, and when the action is minimal, we recover the usual notion of proximal action. In particular, any non-elementary subgroup $G$ which is proximal in restriction to the minimal invariant subset is (continuously) semi-conjugate to a proximal action. As a consequence, if $G$ is proximal in restriction to the minimal invariant subset, then it is automatically non-elementary.
The following fundamental result can be deduced from the work of Antonov \cite{Antonov}, although it has been unknown to experts for a long time (see e.g.\ Ghys \cite[\S 5.2]{ghys-circle}). Our statement is very close to that appearing in the work of Malyutin \cite[Theorem 1]{Malyutin}.

\begin{theorem}[Antonov]\label{t.antonov}
	Let $G\le \homeo_+(\T)$ be a non-elementary subgroup. Then there exists a finite order element $\gamma\in \homeo_+(\T)$ which commutes with every $g\in G$, and such that the induced action of $G$ on the quotient $\T/\langle \gamma\rangle$ is proximal in restriction to the minimal invariant subset.
\end{theorem}

	\begin{remark}\label{rem.simply}
		In the setting of Theorem \ref{t.antonov}, several dynamical properties are equivalent for the actions of $G$ on $\T$ and $\T/\langle \gamma\rangle$:
	\begin{itemize}
		\item (local) discreteness,
		\item minimality,
		\item all non-trivial elements have a uniformly bounded number of fixed points (because if $\gamma$ is of order $d$, and an element $g\in G$ acts on $\T/\langle \gamma\rangle$ with $k$ fixed points, then $g^d$ acts on $\T$ with $dk$ fixed points).
	\end{itemize}
	In particular, the last argument gives that if an element $g\in G$ acts on $\T/\langle \gamma\rangle$ without fixed points, then it also acts without fixed points on $\T$.
	
	We will use this correspondence several times in the rest of the section to work under the more convenient assumption that the action of $G$ on $\T$ is proximal in restriction to the minimal set.
\end{remark}

\begin{lemma}\label{l.non_loc_discrete_entao_non_discrete}\label{r.non-discrete_is_minimal}
	If $G\le\homeo_+(\T)$ is a non-elementary, non-locally discrete subgroup, then $G$ is non-discrete.
	
	Moreover, if $G$ has at most $N$ fixed points, its action on the circle is minimal.
\end{lemma}

\begin{proof}
	Because of the principle in Remark \ref{rem.simply}, we can assume that the action of $G$ on $\T$ is proximal in restriction to the minimal invariant subset.
	Since $G$ is non-locally discrete, there exists a non-wandering open interval $I\subset\T$, and a sequence of non-trivial elements $(g_n)_{n\in \NN}\subset G$, such that $g_n|_{I}\rightarrow \idid|_{I}$. If $\overline I=\T$, we immediately get that $G$ is non-discrete. Otherwise, fix $\eta\in (0,\frac12)$, and a non-wandering open interval $J=(a,b)\subset \T$ with $|J|<\frac{\eta}2$. By proximality in restriction to the minimal invariant subset, we can find an element $f\in G$ such that $f(\T\setminus \overline I)\subset J$. By continuity of composition, we have that the sequence of conjugates $h_n:=fg_nf^{-1}$ converges to the identity in restriction to $f(I)$. In particular, we can fix $n\in \NN$ such that $|h_{n}(x) - x|<\frac{\eta}{2}$ for every $x\in f(I)$. On the other hand, if $x\in J=(a,b)$, as $a,b\in f(I)$, we have
		\[
		a-\frac\eta2 \le h_n(a)\le h_n(x)\le h_n(b)\le b+\frac\eta2,
		\]
	and since $|J|<\frac\eta2$, this gives 
	\[
	-\eta< (a-x)-\frac\eta2 \le h_n(x)-x \le (b-x)+\frac\eta2 <\eta.
	\]
	As $\T=f(I)\cup J$, we conclude that $|h_n(x)-x|<\eta$ for any $x\in \T$, as wanted.
	
	For the second part of the statement, assume the action admits an invariant Cantor set $\Lambda\subset \T$. For any given $\varepsilon>0$, every element $g\in G$ which is $\varepsilon$-close to the identity must fix every gap of $\Lambda$ (that is, any connected component of the complement of $\Lambda$)  whose size exceeds $\varepsilon$. For sufficiently small $\varepsilon$, this gives that $g$ fixes more than $N$ points, and thus $g=\idid$.
\end{proof}

\begin{lemma}\label{l.non-discrete_without_fixed_points}
	Let $G< \homeo_+(\T)$ be a non-elementary, non-discrete subgroup with at most $N$ fixed points. Then, there exists a sequence of fixed-point-free elements in $G$ converging to the identity.
\end{lemma}

\begin{proof}
	After Lemma \ref{r.non-discrete_is_minimal}, we know that the action of $G$ on $\T$ is minimal. Using again the correspondence in Remark \ref{rem.simply}, we can assume that $G$ is proximal.
	Under this additional assumption, we  can take a sequence of non-trivial elements converging to the identity whose fixed points are all contained in an arbitrarily small interval:
	
	\begin{claim}\label{claim.localize}
		For every interval $J\subset \T$ and every $\varepsilon>0$, there exists an element $g\in G$ which is $\varepsilon$-close to the identity and has no fixed points in the interval $J$.
	\end{claim}
	
	\begin{proof}[Proof of claim]
		Let $(g_n)_{n\in \NN}\subset G$ be a sequence of non-trivial elements such that $g_n\to \idid$.  
		If $(g_n)_{n\in \NN}\subset G$ contains a subsequence without fixed points, there is nothing to do. Otherwise, by taking a subsequence, we can assume that the fixed points of $g_{n}$ are converging to points $p_1,\ldots,p_M\in \T$, for some $M\le N$.  Let $J\subset\T$ be any closed interval. By proximality of $G$, there exists an element $k\in G$ which sends $p_1,\ldots,p_M$ to the complement of $J$.  Now, by choosing $g_{n}$ close enough to the identity, we can assume that $kg_{n}k^{-1}\in G$ is $\varepsilon$-close to the identity, and that this element only fixes points which are in the complement of the interval $J$.
	\end{proof}

	Throughout the rest of the proof, we use the distance $d_\infty$ on $\homeo_+(\T)$ to quantify how close elements are. Recall that this can be locally computed by using the uniform norm $\|\cdot\|$.
	We assume by contradiction that for a fixed (small) $\varepsilon>0$,  every element $\varepsilon$-close to the identity has fixed points.  After Claim \ref{claim.localize}, we can define recursively a sequence of nested intervals $( J_n)\subset \T$ with the following properties.
	\begin{enumerate}
		\item[(1)] The sequence $(J_n)$ is shrinking to a point $p\in \T$, namely $\bigcap_n J_n= \{p\}$.
		\item[(2)] For any $n\in \NN$, there exists an element $g_n\in G$ such that  $\|g_n-\idid\|\leq\frac{\varepsilon}{2}$ and $g_n$ is positive on the complement of $J_n$, with $g_n(J_n)=J_n$.
	\end{enumerate}

	\begin{claim}\label{claim.notsoclose}
		For any $n\in \NN$ such that $|J_n|<1-\frac{\varepsilon}{2}$, there exists $m\in \NN$ such that $\frac{\varepsilon}{4}<\|g^m_n-\idid\|\leq\frac{\varepsilon}{2}$.	
	\end{claim}
	
	\begin{proof}[Proof of claim]
		Indeed, if we assume that $g_n$ is $\frac{\varepsilon}{4}$-close to the identity (otherwise $m=1$ works), since $g_n$ has fixed points only in $J_n$, for a sufficiently large power $m\in\NN$, the distance of $g_n^m$ to the identity will be larger than $\frac{\varepsilon}{2}$. Therefore, there exists $m_0\in\NN$ such that $g_n^{m_0}$ is not $\frac{\varepsilon}{2}$-close to the identity, but $g_n^m$ is $\frac{\varepsilon}{2}$-close to the identity for every $0\leq m<m_0$. The point is that $g_n^{m_0-1}$ is $\frac{\varepsilon}{2}$-close to the identity, but it is not $\frac{\varepsilon}{4}$-close: indeed, there exists $x\in\T$ such that
		$$\frac{\varepsilon}{2}<|g_n^{m_0}(x)-x|<\left|g_n\left(g_n^{m_0-1}(x)\right)-g_n^{m_0-1}(x)\right|+|g_n^{m_0-1}(x)-x|<\frac{\varepsilon}{4}+|g_n^{m_0-1}(x)-x|.$$
		This proves the claim.
	\end{proof}
	
	From now on, after Claim \ref{claim.notsoclose}, we can and will assume that $\frac{\varepsilon}{4}<\|g_n-\idid\|\leq\frac{\varepsilon}{2}$. Hence, when $n$ is sufficiently large so that $|J_n|<\frac{\varepsilon}4$, we can find a point  $x_n\in \T\setminus J_n$ such that $g_n(x_n)>x_n+\frac{\varepsilon}{4}$. Consider the interval $I_n\subset \T$ defined by $I_n:=(x_n+\frac{\varepsilon}{12},x_n+\frac{\varepsilon}{6})$. After passing to a subsequence, we can assume that $x_n$ converges to a point $x\in\T$, and for $n_0\in\NN$ large enough we have that $I:=\bigcap_{n\geq n_0}I_n$ is a non-trivial interval.
	With such choices, for every $y\in I$ and $n\geq n_0$ we have $x_n+\frac{\varepsilon}6>y$ and
	\begin{equation}\label{eq.gny}
	g_n(y)>g_n(x_n)>x_n+\frac{\varepsilon}4>y+\frac{\varepsilon}{12}.
	\end{equation}
	On the other hand, for every $n\in \NN$, the interval $J_n$ is disjoint from the interval $[x_n,x_n+\frac{\varepsilon}4]$, so that by choosing $n_1\ge n_0$ such that $|x_n-x|<\frac{\varepsilon}{48}$ for any $n\ge n_1$, we have that the intervals $I$ and $J_n$ are at least $\frac{\varepsilon}{24}$-apart for any such $n$. We deduce that the union $J:=\bigcup_{n\geq n_1} J_n$ and $I$ are at least $\frac{\varepsilon}{24}$-apart. After Claim \ref{claim.localize}, we can take an element $f\in G$ which is $\frac{\varepsilon}{2}$-close to the identity, and negative on the complement of $I$. As $J$ and $I$ are separated, we can find $\delta>0$ such that $f(y)<y-\delta$ for every $y\in J$.
	
	Take $m>n_1$ sufficiently large such that $|J_m|<\delta$; we claim that the element $f^{-1}g_m\in G$ is $\varepsilon$-close to the identity and it has no fixed points in the circle. 
	Indeed, since $f$ is negative on the complement of $I$ and $g_m$ is positive on the complement of $J_m$, it is clear that $g_m$ does not cross $f$ in the complement of $J_m\cup I$. Now, we have $|I|\le |I_m|=\frac{\varepsilon}{12}$, so that from the inequality \eqref{eq.gny} we deduce that $g_m$ does not cross $f$ in $I$. Similarly, the size of the interval $J_m$ is smaller than $\delta$ and $f(y)<y-\delta$ for every $y\in J\supset J_m$, which implies that $g_m$ does not cross $f$ in $J_m$. Therefore, $g_m$ does not cross $f$ in the whole circle $\T$, which implies that the element $f^{-1}g_m$ has no fixed points in $\T$. Finally, as $f^{-1}g_m$ is the composition of two elements $\frac{\varepsilon}{2}$-close to the identity, this element is $\varepsilon$-close to the identity, as desired.
\end{proof}
\setcounter{claim}{0}

Given a subgroup $G\leq \mathrm{Homeo}_+(\mathbb{S}^1)$, we denote by $\overline G$ its closure in $\mathrm{Homeo}_+(\mathbb{S}^1)$ with respect to the $C^0$ topology, which is still a subgroup of $\mathrm{Homeo}_+(\mathbb{S}^1)$.

\begin{lemma}\label{l.locally_globally}
	Let $G< \homeo_+(\T)$ be a non-elementary, non-discrete subgroup with at most $N$ fixed points. Then, its closure $\overline G$ has at most $N$ fixed points.
\end{lemma}
\begin{proof}
	As before, we use the correspondence in Remark \ref{rem.simply} to work under the additional assumption that $G$ is proximal.
	We then remark that any $f\in \overline{G}$ has at most $\lfloor\frac{N}2\rfloor$ components $I$ of $\supp(f):=\T\setminus \fix(f)$ for which $f(x)>x$ for every $x\in I$. Indeed, arguing by contradiction, by Lemma~\ref{l.non-discrete_without_fixed_points}, we can take an element $g\in G$, sufficiently close to the identity and without fixed points, so that $g$ crosses weakly hyperbolically the element $f$ at least $2\left (\lfloor\frac{N}2\rfloor +1\right )\ge N+1$ times. After Remark \ref{rem:whyp_stable}, we can find an element $f_0\in G$ sufficiently close to $f$, which also crosses hyperbolically $g$ at least $N+1$ times. This contradicts the assumption that $G$ have at most $N$ fixed points. Repeating the argument for $f^{-1}$, we deduce that $\supp(f)$ has at most $2\lfloor\frac{N}2\rfloor\le N$ connected components, and so does $\fix(f)$. 
	
	Assume now that $\overline G$ contains a non-trivial element $f$ such that $\fix(f)$ contains a non-empty open interval $I$. Let us choose such an element $f\in \overline G$ such that the number $C$ of connected components of $\supp(f)$ is maximal.	As we are assuming that $G$ be proximal, we can find an element $h\in G$ such that $\overline{h(\supp(f))}\subset I$. Then the element $f'=fhfh^{-1}\in \overline{G}$ is such that $\supp(f')$ has $2C>C$ connected components, contradicting maximality.
\end{proof}

The next result we need is very general (and classical). 

\begin{lemma}\label{l.contruction_irrational_rotation_number}
	Let $f\in\homeo_+(\T)$ be a circle homeomorphism of order $q\geq2$, with rotation number $\rot(f)=\frac{p}{q}$, where $p\ge 1$ is a natural number. Then, for every $\varepsilon>0$ there exists $\delta\in (0,1/2)$ such that for every circle homeomorphism $g\in\homeo_+(\T)$ which is positive and $\delta$-close to the identity, one has
	$$d_{\mathcal{C}^0}(f,gf)<\varepsilon\quad \text{and} \quad \rot(gf)\in \left(\frac{p}{q},\frac{p}{q}+\frac{1}{q^3}\right].$$
\end{lemma}
\begin{proof}
	Note that the homeomorphism $f$ is conjugate in $\homeo_+(\T)$ to the rotation $R_{p/q}$ (see Herman \cite[\S II.6]{Herman})
	Since $\homeo_+(\T)$ is a topological group (so composition, and thus conjugacy, are continuous), and since the function rotation number is continuous on $\homeo_+(\T)$, and conjugacy-invariant, we can assume $f=R_{p/q}$, and fix $\varepsilon>0$ and $\delta>0$ such that
		\begin{equation}\label{eq:rot_pert}
			d_{\mathcal{C}^0}(R_{p/q},gR_{p/q})<\varepsilon\quad \text{and} \quad \rot(gR_{p/q})\in \left(\frac{p}{q}-\frac{1}{q^3},\frac{p}{q}+\frac{1}{q^3}\right]
		\end{equation}
	for every circle homeomorphism $g\in\homeo_+(\T)$ which is $\delta$-close to the identity.
	
	Let us next make the basic observation that if $h,k\in \homeo_+(\RR)$ are two homeomorphisms of the real line such that $h\le k$ (meaning that $h(x)\le k(x)$ for any $x\in \RR$), then for any integer $n\ge0$, one has $h^n\le k^n$. This can be easily checked by induction:
		\[h^n(x)=h(h^{n-1}(x))\le h(k^{n-1}(x))\le k(k^{n-1}(x))=k^n(x).\]
	For homeomorphisms $h\in \homeo_{\ZZ}(\RR)$ of the real line commuting with integer translations, the \emph{translation number} $\tau(h)$ is defined as
		\[
		\tau(h)=\lim_{n\to\infty}\frac{h^n(x)}{n},
		\]
	and this limit does not depend on the choice $x\in \RR$. Consequently, if $h,k\in \homeo_{\ZZ}(\RR)$ are such that $h\le k$, then $\tau(h)\le \tau(k)$. The rotation number $\rot(f)$ of a homeomorphism $f\in \homeo_+(\T)$ is classically defined by taking any lift $\widetilde f\in \homeo_{\ZZ}(\RR)$ and setting $\rot(f)=\tau(\widetilde f)\pmod\ZZ$.
	
	For what follows, we denote by $T_\alpha$ the translation by $a\in \RR$. Lift $f=R_{p/q}$ to the translation $T_{p/q}$.  If $g\in \homeo_+(\T)$ is positive and $\delta$-close to the identity, we can find some $\lambda\in (0,\delta)$ and a lift $\widetilde g\in \homeo_{\ZZ}(\RR)$ such that $T_{p/q}<T_{p/q+\lambda}\le \widetilde gT_{p/q}\le T_{p/q+\delta}$. Taking translation numbers, we get
	\[
	p/q+\lambda\le \tau(\widetilde gR_{p/q})\le p/q+\delta,
	\]
	so that passing to rotation numbers, and considering condition \eqref{eq:rot_pert}, we get
	\[
	\rot(gR_{p/q})\in \left[\frac{p}q+\lambda,\frac{p}{q}+\frac{1}{q^3}\right],
	\]
	as desired.	
\end{proof}

\begin{lemma}\label{l.exists_irrational}
	Let $G< \homeo_+(\T)$ be a non-elementary, non-discrete subgroup with at most $N$ fixed points. Then, $\overline{G}$ contains an element with irrational rotation number.
\end{lemma}
\begin{proof}
	If the subgroup $G<\homeo_+(\T)$ has an element with irrational rotation number there is nothing to prove, because $G\subset \overline{G}$. Therefore, we will suppose that $G$ has no element with irrational rotation number. We want to construct a converging sequence $\left(h_n\right)_{n\in\NN}\subset G$ whose limit $h\in\overline{G}$ has irrational rotation number.
	
	We start by choosing a sequence of elements $\left(f_n\right)_{n\in\NN}\subset G$ without fixed points and converging to the identity, whose existence is ensured by Lemma~\ref{l.non-discrete_without_fixed_points}. After changing $f_n$ for $f_n^{-1}$ when necessary and taking a subsequence, we can assume that $\left(f_n\right)_{n\in\NN}$ is a sequence of positive circle homeomorphisms whose distance to the identity decreases. Let us fix also a sequence $(\varepsilon_n)_{n\in \NN}$ of positive numbers such that $\sum \varepsilon_n<1/2$.	We choose $h_0\in G$ to be the first element $f_{m_0}$ such that $d_{\mathcal{C}^0}(f_{m_0},\idid)<\varepsilon_0$. We write $\rot(h_0)=p_0/q_0$, where $p_0, q_0\ge 1$ are natural numbers with $\gcd(p_0,q_0)=1$.
	
	Now, let us assume by induction that $h_n\in G$ is a positive circle homeomorphism such that $d_{\mathcal{C}^0}(h_n,\idid)<\sum_{k=0}^{n}\varepsilon_k$ and of rational rotation number \[\rot(h_n)=\frac{p_n}{q_n}\in \left(\frac{p_{n-1}}{q_{n-1}},\frac{p_{n-1}}{q_{n-1}}+\frac{1}{q_{n-1}^3}\right],\]
	where $p_n, q_n\ge 1$ are natural numbers with $\gcd(p_n,q_n)=1$.
	As $G$ has at most $N$ fixed points, whenever $q_n>N$ (this condition is satisfied for any sufficiently large $n\in \NN$), we must have $h_n^{q_n}=\idid$.
	By Lemma~\ref{l.contruction_irrational_rotation_number}, there exists $\delta_n>0$ such that for every  positive circle homeomorphism $g\in\homeo_+(\T)$ which is $\delta_n$-close to the identity, we have that
	$d_{\mathcal{C}^0}(h_n,gh_n)<\varepsilon_{n+1}\, \, \text{ and } \, \, \rot(gh_n)\in \left(\frac{p_n}{q_n},\frac{p_n}{q_n}+\frac{1}{q_n^3}\right]$. Then, we define $h_{n+1}$ as $f_{m_n}h_n\in G$, where $f_{m_n}$ is the first element of the sequence $\left(f_n\right)_{n\in\NN}$ such that $d_{\mathcal C^0}(f_{m_n},\idid)<\delta_n$. This gives that
	$$d_{\mathcal{C}^0}(h_{n+1},\idid)\leq d_{\mathcal{C}^0}(h_{n+1},h_n)+ d_{\mathcal{C}^0}(h_{n},\idid)<d_{\mathcal{C}^0}(f_{m_n}h_{n},h_n)+\sum_{k=0}^{n}\varepsilon_k<\sum_{k=0}^{n+1}\varepsilon_k.$$
	Since $\sum_{k=0}^{n+1}\varepsilon_k<\frac{1}{2}$, $h_{n+1}\in G$ is a positive circle homeomorphism (Remark \ref{rem.composition_positive}). By our assumption on $G$, the rotation number of $h_{n+1}$ is rational, and we write $\rot(h_{n+1})=\frac{p_{n+1}}{q_{n+1}}$, which belongs to $\left(\frac{p_n}{q_n},\frac{p_n}{q_n}+\frac{1}{q_n^3}\right]$ by construction. Hence, the inductive assumptions are satisfied.
	
	We next check that $(h_n)$ is a Cauchy sequence with respect to the $d_{\mathcal C^0}$-distance (for which $\homeo_+(\T)$ is complete), so that it admits a limit $h\in \overline{G}$. Indeed, for every $n,m\in\NN$ we have
	$$d_{\mathcal{C}^0}(h_{n+m},h_n)\leq\sum_{k=n}^{n+m}d_{\mathcal{C}^0}(h_{k+1},h_k)<\sum_{k=n}^{n+m}\varepsilon_k<\sum_{k=n}^{+\infty}\varepsilon_k\xrightarrow{\,\,\,n\rightarrow+\infty\,\,\,} 0.$$
	It remains to prove that $\rot(h)\notin \mathbb{Q}$. To see this, for every $n\in\NN$, we consider the interval $I_n=\left(\frac{p_{n}}{q_{n}},\frac{p_{n}}{q_{n}}+\frac{1}{q_{n}^2}\right)$. By the classical Dirichlet's approximation theorem, we have that $I_n$ is a nested sequence of intervals, such that $\bigcap I_n$ is an irrational number. Let us detail this for completeness. First, we recall that $$\frac{p_{n+1}}{q_{n+1}}\in\left(\frac{p_n}{q_n},\frac{p_n}{q_n}+\frac{1}{q_n^3}\right],$$
	so that $\frac{p_{n+1}}{q_{n+1}}> \frac{p_n}{q_n}$. Next, we want to prove that $\frac{p_{n+1}}{q_{n+1}}+\frac{1}{q_{n+1}^2}< \frac{p_{n}}{q_{n}}+\frac{1}{q_n^2}$. We will first show that $q_{n+1}>q_n$. Indeed, if $q_{n+1}\leq q_n$ then
	$$p_{n+1}\in\left(\frac{p_nq_{n+1}}{q_n},\frac{p_nq_{n+1}}{q_n}+\frac{q_{n+1}}{q_n^3}\right]\subset \left(\frac{p_nq_{n+1}}{q_n},\frac{p_nq_{n+1}}{q_n}+\frac{1}{q_n^2}\right],$$ 
	which is an absurd, because $\left(\frac{p_nq_{n+1}}{q_n},\frac{p_nq_{n+1}}{q_n}+\frac{1}{q_n^2}\right]$ does not contain any integer. Now, we have the following inequality 
	$$\frac{p_{n+1}}{q_{n+1}}+\frac{1}{q_{n+1}^2}\leq\frac{p_n}{q_n}+\frac{1}{q_n^3}+\frac{1}{q_{n+1}^2}\leq \frac{p_n}{q_n}+\frac{1}{q_n^3}+\frac{1}{(q_{n}+1)^2}=\frac{p_n}{q_n}+\frac{q_n^3+(q_n+1)^2}{q_n^3(q_n+1)^2}$$
	$$=\frac{p_n}{q_n}+\frac{1}{q_n^2}\frac{q_n^3+q_n^2+2q_n+1}{q_n^3+2q_n^2+q_n}<\frac{p_n}{q_n}+\frac{1}{q_n^2}\frac{q_n^3+q_n^2+2q_n+1+(q_n^2-q_n-1)}{q_n^3+2q_n^2+q_n}=\frac{p_n}{q_n}+\frac{1}{q_n^2}.$$
	So, we conclude that $I_{n+1}\subset I_n$ for every $n\in\NN$. On the other hand, $\left(q_n\right)_{n\in\NN}$ is an increasing sequence of integers, and therefore $|I_n|=\frac{1}{q_n^2}\longrightarrow 0$. Hence, $\bigcap_{n\in\NN}I_n$ converges to a point $\alpha\in\RR$. We claim that $\alpha$ is not a rational number, otherwise $\alpha=\frac{p}{q}$ with $p,q\in\ZZ$, and $q\geq 2$, which implies that $\frac{p}{q}\in I_n=\left(\frac{p_n}{q_n},\frac{p_n}{q_n}+\frac{1}{q^2_n}\right)$, for every $n\in\NN$. Therefore, for $n$ sufficiently large, we have $2q<q_n$ and then 
	$$p\in \left(\frac{p_nq}{q_n},\frac{p_nq}{q_n}+\frac{q}{q^2_n}\right)\subset \left(\frac{p_nq}{q_n},\frac{p_nq}{q_n}+\frac{1}{2q_n}\right),$$
	which is an absurd, because $\left(\frac{p_nq}{q_n},\frac{p_nq}{q_n}+\frac{1}{2q_n}\right)$ does not contain any integer.
	
	Let us go back to our converging sequence $(h_n)$. By construction we have
	\[
	\rot(h_{n+1})\in \left(\frac{p_n}{q_n},\frac{p_n}{q_n}+\frac{1}{q_n^3}\right]\subset \left(\frac{p_n}{q_n},\frac{p_n}{q_n}+\frac{1}{q_n^2}\right)=I_n,
	\]
	so that $\rot(h_{n+k})\in I_{n+k}\subset I_n$ for any $k\ge 1$. Hence, by  continuity of $\rot:\homeo_+(\T)\rightarrow\T$, it follows that $\rot(h)\in I_n$, for every $n\in\NN$, and therefore $\rot(h)=\alpha$ is irrational, as desired.
\end{proof}

\begin{lemma}\label{l.contains_irrational}
	Let $G< \homeo_+(\T)$ be a non-elementary, non-discrete subgroup with at most $N$ fixed points. Then, $\overline{G}$ contains an element which is conjugate to an irrational rotation. Therefore, up to conjugacy, $\overline{G}$ contains $\SO(2)$.
\end{lemma}
\begin{proof}
	We will prove that any $f\in\overline{G}$ with irrational rotation number $\alpha=\rot(f)$ (by Lemma \ref{l.exists_irrational}, we can find such an element in $G$) is conjugate to the rotation $R_\alpha$. We argue by way of contradiction. Let us denote by $\Lambda$ the minimal $f$-invariant Cantor set.	Recall that a gap of $\Lambda$ is a connected component of the complement $\T\setminus \Lambda$. Up to conjugating $G$ by some circle homeomorphism, we can assume that the size of any gap does not exceed $\frac14$.
	After the discussion in Section \ref{s.preliminaries}, $f$ is semi-conjugate to the rotation $R_\alpha$. Denoting by $(q_n)$ the sequence of denominators of rational approximations of $\alpha$, we have $R_\alpha^{q_n}\to \idid$, which implies that for any $x\in\Lambda$ which is not in the closure of any wandering interval, we have $f^{q_n}(x)\to x$. (This is because any such $x$ satisfies that $\{x\}=h^{-1}(h(x))$, where $h:\T\to \T$ is a continuous monotone map such that $hf=R_\alpha h$ giving the semi-conjugacy.) After the choice of the size of the gaps, we can take $n_0\in \NN$ sufficiently large so that $d_\infty(f^{q_n},\idid)<1/2$ for any $n\ge n_0$, and then take a subsequence $(q_{n_j})_{j\in \NN}\subset (q_n)_{n\ge n_0}$ such that $f^{q_{n_j}}$ is positive for any $j\in \NN$.
	
	Take $\varepsilon>0$ smaller than the size of the $N$ largest gaps. By Lemma \ref{l.non-discrete_without_fixed_points}, we can choose a positive element $g\in G$ which is $\varepsilon$-close to the identity. Note that for any $j\in\NN$ sufficiently large, $g$ crosses $f^{q_{n_j}}$ twice on a small neighborhood of any of the $N$ largest gaps, hence at least $2N$ times. This contradicts the assumption that $g^{-1}f^{n_j}\in\overline{G}$ have at most $N$ fixed points. So we conclude that $f\in\overline{G}$ with $\rot(f)\notin\QQ$ is conjugate to an irrational rotation and therefore $\overline{G}$ contains a conjugate copy of $\SO(2)$.
\end{proof}

We finally use a result by Giblin and Markovic \cite[Theorem 1.2]{GiblinMarkovic}. (A self-contained argument for the conclusion in the case $N=2$ is given in the second author's PhD thesis \cite{theseJoao}.)

\begin{theorem}[Giblin--Markovic]\label{t.giblin-markovic}
	Let $G\le \homeo_+(\T)$ be a closed transitive subgroup which contains a non-constant continuous path. Then we have the following alternative:
	\begin{enumerate}
		\item either $G$ is conjugate to $\SO(2)$, or
		\item $G$ is conjugate to $\psl^{(k)}(2,\RR)$, for some $k\ge 1$, or
		\item $G$ is conjugate to $\homeo_+^{(k)}(\T)$, for some $k\ge 1$, where $\homeo_+^{(k)}(\T)$ is the group of all homeomorphisms commuting with the group of order $k$ rotations.
	\end{enumerate}
\end{theorem}

We can now put everything together and prove the main result of our work.

\begin{proof}[Proof of Theorem \ref{mainthmA.non-discrete}]
	Let $G$ be a non-elementary subgroup with at most $N$ fixed points. If $G$ is non-locally discrete, then after Lemma \ref{l.non_loc_discrete_entao_non_discrete}, $G$ is non-discrete, so that by Lemma \ref{l.contains_irrational} its closure $\overline G$ contains a conjugate copy of the subgroup of rotations $\SO(2)$. In particular, $\overline G$ is closed, transitive and contains a non-constant continuous path. As $G$ is non-elementary, the first possibility in Theorem \ref{t.giblin-markovic} cannot occur. On the other hand, after Lemma \ref{l.locally_globally}, $\overline G$ has at most $N$ fixed points, so that the third possibility in Theorem \ref{t.giblin-markovic} cannot occur either. We conclude that $\overline G$ is conjugate to $\psl^{(k)}(2,\RR)$, for some $k\ge 1$, as desired.
\end{proof}

{\small \subsection*{Acknowledgments}
	This work is based on part of the second author's PhD thesis, and we thank the readers Andrés Navas and Maxime Wolff, as well as the anonymous referee, for their careful reading and suggestions. We are particularly grateful to Maxime Wolff for pointing out that the result of Giblin and Markovic could shortcut our proof.
	The authors acknowledge the support of the project MATH AMSUD, DGT -- Dynamical Group Theory (22-MATH-03) and the project ANR
	Gromeov (ANR-19-CE40-0007), and the host department IMB
	receives support from the EIPHI Graduate School (ANR-17-EURE-0002). M.T. has been partially supported by the project ANER
	Agroupes (AAP 2019 Région Bourgogne--Franche--Comté).}

\bibliographystyle{alpha}
\bibliography{biblio}

\medskip

\noindent{Institut de Mathématiques de Bourgogne (IMB, UMR CNRS 5584)\\
	Université de Bourgogne \\
	9 av.\ Alain Savary, 21000 Dijon, France\\}
\noindent \texttt{christian.bonatti@u-bourgogne.fr}\\
\noindent \texttt{joao.carnevale@hotmail.com}\\
\noindent \texttt{michele.triestino@u-bourgogne.fr}\\

 \end{document}